\documentclass [a4paper, 12t]{article}

\usepackage{amsmath}
\usepackage{amsfonts}
\usepackage{amssymb}
\usepackage{amsthm}
\usepackage{makeidx}
\usepackage{graphicx}
\usepackage{pb-diagram}
\usepackage{epstopdf}
\usepackage{fancyhdr}
\usepackage{enumitem}
\usepackage[all]{xy}

\newtheorem{teo}{Theorem}
\newtheorem{prop}[teo]{Proposition} 

\newtheorem{cor}[teo]{Corollary}
\theoremstyle{definition}

\newtheorem{nota}[teo]{Note}

\newtheorem{con*}{Conjecture}

\begin{document}
\setlength{\parskip}{1ex plus 0.5ex minus 0.2ex}
\begin{center}
\textbf{ON A CLASS OF C*-PREDUALS OF $l_1$}\\
\end{center}
\begin{center}
\textsl{By STEFANO ROSSI}
\end{center}
$$$$\\
$$$$
\textbf{Abstract}\quad
As it is well known, the Banach space $l_1$ of absolutely summable (complex) sequences endowed with the $\|\cdot\|_1$ norm is not \emph{unique predual}. This means that there are many different (\emph{i.e.} non isometrically isomorphic) Banach spaces $X$ such that $X^*\cong l_1$.\\
The present note is aimed to point out a simple class of C*- preduals of $l_1$: namely the spaces $C_{\tau}(\mathbb{N})$ of continuous functions $f:\mathbb{N}\rightarrow\mathbb{C}$, where the set of natural numbers $\mathbb{N}$ is equipped with a compact Hausdorff topology $\mathcal{T}$.\\
To be more concrete, we shall explicitly describe a countable collection $\{\mathcal{T}_n\}$ of such topologies.\\
Finally, we also provide an abstract characterization of the previous preduals as closed subspaces $M\subset l^{\infty}$ rich of positive elements.

$$$$
As commonly used in the literature, we shall denote by $l_1$ the (complex) Banach space of absolutely summable sequences, given of the norm $\|\cdot\|_1$  defined by $\|a\|_1\doteq\sum_{i=1}^{\infty}|a_i|$ for each $a\in l_1$.\\
It is a very well known fact that $l_1$ is a conjugate Banach space, that is there exists at least a Banach space $X$, such that $X^*\cong l_1$ (isometric isomorphism). Such a space is usually named a \emph{predual}. The most famous predual of $l_1$ is probably represented by the space $c_0$ of those (complex) sequences converging to $0$, endowed of the \emph{sup}-norm. In this case, the isometric isomorphism $c_0^*\cong l_1$ is the map $\Psi: l_1\rightarrow c_0^*$ given by $\langle\Psi(y), x\rangle\doteq\sum_{i=1}^{\infty}y_ix_i$ for every $x\in c_0$ and $y\in l_1$.\\
In spite of its simple definition, $l_1$ is a rather pathological\footnote{The weak topology of $l_1$ is not well behaved: every weakly convergent sequence is indeed norm-convergent, although the weak topology is strictly weaker than the norm topology. } Banach space: for instance the predual is not unique; there is in fact a plenty of (non isomorphic) preduals of $l_1$. Some of these are quite "irregular": Y. Benyamini and J. Lindenstrauss \cite{Linde} proved in 1972 that there is a predual of $l_1$ that is not (topologically) complemented in any $C(K)$-space, $K$ being any compact Hausdorff topological space.\\
On the other hand, the present paper is aimed to discuss a very nice class of C*-preduals of $l_1$. In this spirit, the first thing that should be noticed is the following:

\begin{prop}
\label{prop1}
If $\mathcal{T}$ is a compact Hausdorff topology on the set of natural numbers $\mathbb{N}$, one has $C_{\tau}(\mathbb{N})^*\cong l_1$.
\end{prop}
\begin{proof}
It is possible to prove the statement by using the Riesz-Markov theorem. Here we perform a proof based on the characterization of separable conjugate spaces given in \cite{Rossi}. To this aim, we only have to check that $C_{\tau}(\mathbb{N})\subset l^{\infty}$ is a closed, norm-attaining and $1$-norming subspace.\\
$C_{\tau}(\mathbb{N})$ is closed in $l^{\infty}$ as a complete subspace. It is norm-attaining (when it is thought as subspace of bounded linear functionals on $l_1$ ) thanks to Weierstrass' theorem, since $(\mathbb{N},\mathcal{T})$ is a compact space by assumption.\\ If $y\in l_1$ and $\varepsilon>0$, there is $n\in\mathbb{N}$ such that $\|y\|_1\leq \sum_{i=1}^n|y_i|+\varepsilon$. Let $\theta_i\in\mathbb{R}$ such that $y_i=|y_i|e^{i\theta_i}$ for each $i=1,2\dots,n$.
The subset $C_n\doteq\{1, 2, \dots,n\}\subset\mathbb{N}$ is closed (and discrete), hence the function $f:C_n\rightarrow\mathbb{C}$ given by $f(i)=e^{-i\theta_i}$ for each $i\in C_n$ is continuous and $\|f\|_{\infty}=1$. Since $(\mathbb{N},\mathcal{T})$ is a compact Hausdorff space, it is a normal topological space, so Tietze extension theorem applies to get a function $g\in C_{\tau}(\mathbb{N})$ such that $\|g\|_{\infty}=1$ and $g(i)=e^{-i\theta_i}$ for each $i\in{1, 2,\dots,n}$.\\
We have $|\langle g, y\rangle|=|\sum_{i=1}^{\infty} g(i)y_i|\geq\sum_{i=1}^n |y_i|-\varepsilon\geq \|y\|_1-2\varepsilon$. The last inequality easily implies
that $$\sup_{g\in C_{\tau}(\mathbb{N})_1}|\langle g, y\rangle|=\|y\|_1$$
that is $C_{\tau}(\mathbb{N})\subset l^{\infty}$ is a $1$-norming subspace. This ends the proof.
\end{proof}
The previous proposition immediately leads to the following  corollary in point-set topology:
\begin{cor}
Every compact Hausdorff topology on the set of natural numbers $\mathbb{N}$ is
metrizable.
\end{cor}
\begin{proof}
Let $\mathcal{T}$ be such a topology. We have $C_{\tau}(\mathbb{N})^*\cong l_1$, hence $C_{\tau}(\mathbb{N})$ is a separable Banach space, as a predual of the separable Banach space $l_1$, so that $(\mathbb{N},\mathcal{T})$ is metrizable.
\end{proof}
\begin{nota}
As far as I know, a simple proof of the corollary quoted above does not seem available in the general setting of point-set topology, since it is not apparent that a compact Hausdorff topology on $\mathbb{N}$ is automatically second countable.\\
On the other hand, non first countable topologies on $\mathbb{N}$ are known: \emph{Appert} topology, for instance, provides an elegant example of such a space.
For the reader's convenience, we recall here  that Appert's topology on $\mathbb{N}$ is defined as follows: a subset $A\subset\mathbb{N}$ is open if $1\notin A$ or
(when $1\in A$) if $$\lim_{n\to\infty}\frac{N(n,A)}{n}=1$$
where $N(n,A)\doteq \left|\left\{k\in A: k\leq n\right\}\right|.$\footnote{$|X|$ is the cardinality of any set $X$.}
Appert space is Lindel\"{o}f, separable but it is not first countable, since $1$ does not have a countable basis of neighborhoods. For more details, we refer the interested reader to \cite{Steen} or directly to the original paper by Appert \cite{Appert}.
\end{nota}
Here below we shall describe explicitly a countable collection of compact Hausdorff topologies on $\mathbb{N}$. Before introducing the announced topologies, one should mention that every set $X$ can be endowed with a compact Hausdorff topology, by virtue of a straightforward application of the  Axiom of Choice\footnote{The discrete topology $\mathcal{P}(X)$ on $X$ is locally compact and Hausdorff. The Alexandroff compactification $\hat{X}$ of $X$ is compact and Hausdorff; moreover, if $X$ is an infinite set, there is a bijection $\Phi:X\rightarrow\hat{X}$. We can use $\Phi$ to define a compact Hausdorff topology $\mathcal{T}$ on $X$, by requiring a set $U\subset X$ to be open if $\Phi(U)$ is an open subset of $\hat{X}$.}. \\
Now let $n\in \mathbb{N}$ be a fixed natural number. Given any $k\in\{1,2,\dots, n\}$, we define the sets $A_{k,l}\doteq\{k, mn+k: m\geq l\}$. The sets $A_{k,l}$ allow us to define a topology $\mathcal{T}_n$, whose basis $\mathcal{B}_n$ is given by the subset $B\subset\mathbb{N}$ of the form $A_{k,l}$ if $k\in B$ for some $k\in\{1,2,\dots, n\}$, otherwise we do not put any restriction, namely if $\{1,2\dots, n\}\cap B=\emptyset$ then $B$ is allowed to be any subset of the natural numbers.\\
Since $A_{k,l}\cap A_{k,h}=A_{k,l\vee h}$\footnote{Here $l\vee h$ stands for $\max\{l,h\}$.} and
$A_{k,l}\cap A_{k',h}=\emptyset$ when $k,k'\in\{1,2,\dots, n\}$ are different, $\mathcal{B}_n$ is really a basis.
It is a straightforward verification to check that $\mathcal{T}_n$ is a compact Hausdorff topology; the notion of convergence inherited by this topology  is clearly the following:\\
a sequence $\{n_m: m\in\mathbb{N}\}$ of integers converges to $k\in\{1,2,\dots, n\}$ iff $n_m$ is eventually in a set $A_{k,l}$, while converges to $k>n$ iff it is eventually equal to $k$.\\
In the topology $\mathcal{T}_n$ the set $\{k: k\leq n\}$ is composed by non isolated points, while all the integers $k>n$ are isolated. In some sense, topologies $\mathcal{T}_n$ are as best as possible among compact Hausdorff ones, since it is a straightforward application of \emph{Baire} category theorem that a compact Hausdorff topology on $\mathbb{N}$ cannot have an infinite set of accumulation points\footnote{Whenever $\mathcal{T}$ is a compact Hausdorff topology on $\mathbb{N}$, $(\mathbb{N},\mathcal{T})$ is a Baire space as a complete metric space, hence it cannot be written as a countable union of rare sets, but every  non isolated point $n\in\mathbb{N}$ gives a rare singleton $\{n\}$. In particular, the set of natural numbers $\mathbb{N}$ cannot be given of a connected compact Hausdorff topology; anyway a connected Hausdorff topology on $\mathbb{N}$ is available: for instance\emph{Golomb} topology, see \cite{Golomb}. }.\\
However, what is more important here is that a simple argument can be performed to prove that the topologies $\mathcal{T}_n$ are not homeomorphic:
\begin{prop}
With the notations above, if $n\neq m$ the topological spaces $(\mathbb{N},\mathcal{T}_n)$ and $(\mathbb{N},\mathcal{T}_m)$ are not homeomorphic.
\end{prop}
\begin{proof}
Let us suppose that $m>n$ and let $\Phi:(\mathbb{N},\mathcal{T}_m)\rightarrow (\mathbb{N},\mathcal{T}_n)$ be a continuous injective map.
If $k\in\{1,2,,\dots m\}$, we can consider a sequence $\{n_l\}$ converging to $k$.
The sequence $\{\Phi(n_l)\}$ converges to $\Phi(k)$ thanks to the continuity of $\Phi$. Since $\{n_l\}$ is not constant and $\Phi$ is an injection
$\Phi(k)$ is forced to be a natural number belonging to the subset $\{1,2,\dots,n\}$, against the injectivity of $\Phi$.
\end{proof}
Let us denote by $X_n$ the Banach space $C_{\tau_n}(\mathbb{N})$. Clearly we have $X_n^*\cong l_1$ and
\begin{prop}
If $n\neq m$ the Banach space $X_n$ and $X_m$ are $l_1$-preduals, which are not isometrically isomorphic.
\end{prop}
\begin{proof}
If they were isometrically isomorphic, the topological space $(\mathbb{N},\mathcal{T}_n)$ and $(\mathbb{N},\mathcal{T}_m)$ should be homeomorphic according to the classical Banach-Stone theorem.
\end{proof}

The remaining part of the present paper is devoted to provide an intrinsic characterization of the spaces $C_{\tau}(\mathbb{N})$ as suitable subspaces of
$l^{\infty}$. To this aim, one probably has to remind that any predual $M$ of a conjugate spaces $X$ should be sought as a closed subspace of the dual space $X^*$, which is $1$-\emph{norming}\footnote{A subspace $M\subset X^*$ is said to be $1$-norming if for each $x\in X$, one has $$\|x\|=\sup\{|\varphi(x)|: \varphi\in M_1\}$$ $M_1$ being the unit ball of $M$.} and \emph{norm-attaining}, namely each linear functionals belonging to the subspace is required to attain its norm on the unit ball of $X$. \\
When $X$ is a separable conjugate space, the conditions above are also sufficient for a closed subspace $M\subset X^*$ to be canonically a predual of $X$ as it is shown in \cite{Rossi}.\\
Here canonically means that the isometric isomorphism $X\cong M^*$ is nothing but the restriction of the canonical injection $j:X\rightarrow X^{**}$ to $M$.\\
Before stating the result announced, let us fix some notations: $e\in l^{\infty}$ is the sequence constantly equal to $1$, $M_+$ stands for the positive\footnote{An element $x\in l^{\infty}$ is said to be positive if $x_i\geq 0$ for each $i\in \mathbb{N}$; in this case one writes $x\geq 0$.} cone of a subspace $M\subset l^{\infty}$, while $a^{\frac{1}{2}}$ is the square root\footnote{If $x\geq 0$, then $x^{\frac{1}{2}}$ is the positive sequence given by $x^{\frac{1}{2}}(i)\doteq x_i^{\frac{1}{2}}$ for each $i\in\mathbb{N}$.} of a positive element $a\in l^{\infty}_+$.\\
According to the next theorem the spaces $C_{\tau}(\mathbb{N})$ are precisely those $l_1$-predual rich of positive elements:
$$$$
\begin{teo}
\medskip
Let $M\subset l^{\infty}$ be a predual of $l_1$, such that:\\
$(a)$ $e\in M$.\\
$(b)$ $M_+$ is weakly*-dense in $l^{\infty}_+$.\\
\medskip $(c)$ If $x\in M_+$, then $x^{\frac{1}{2}}\in M_+$.\\
Then $M=C_{\tau}(\mathbb{N})$ for a suitable compact Hausdorff topology on the set of natural numbers $\mathbb{N}$.
\end{teo}

\begin{proof}
Let be $\mathfrak{A}\subset l^{\infty}$ be the unital C*-algebra\footnote{For a basic treatment of $C^*$-algebras theory, we refer the reader to \cite{Davidson}.} generated by $M$. If $\omega$ is a \emph{pure} (multiplicative) state on $\mathfrak{A}$, we can consider its restriction $\omega\upharpoonright_M$. Since $M^*\cong l_1$, we have $\omega(x)=\varphi_y(x)\doteq \sum_i y_ix_i$ for each $x\in M$, where $y$ is a suitable sequence in $l_1$. Now pick a positive element $a\in l_{\infty}$. Thanks to $(b)$, there is a sequence $\{x_n\}_{n\in\mathbb{N}}\subset M_+$ such that $x_n\rightharpoonup a$ (in the weak* topology of $l^{\infty}$). Then we have
\begin{eqnarray}
\varphi_y(a)=\lim_n\varphi_y(x_n)=\lim_n\varphi\left(x_n^{\frac{1}{2}}x_n^{\frac{1}{2}}\right)=\nonumber\\
\lim_n\omega\left(x_n^{\frac{1}{2}}x_n^{\frac{1}{2}}\right)=
\lim_n\omega\left(x_n^{\frac{1}{2}}\right)\omega\left(x_n^{\frac{1}{2}}\right)=
\varphi_y(a^{\frac{1}{2}})^2\nonumber\\
\nonumber
\end{eqnarray}
where the last equality holds since $x_n^{\frac{1}{2}}\rightharpoonup a^{\frac{1}{2}}$ (the weak* convergence in $l^{\infty}$ is nothing but the bounded pointwise convergence).\\
If $e_i\in l^{\infty}$ is the sequence given by $e_i(k)=\delta_{i,k}$, we get $\varphi_y(e_i)=\varphi_y(e_i)^2$, because $e_i^{\frac{1}{2}}$ is $e_i$ itself. It follows that, for each $i\in\mathbb{N}$, $\varphi_y(e_i)$ is $0$ or $1$. Since $\sum_i |y_i|= \|\varphi_y\|=1$, one has $y=e_k$ for some $k$. It easily follows that $\omega$ is the evaluation map at $k$.\\
This means that $\sigma(\mathfrak{A})\cong\mathbb{N}$, hence $\mathfrak{A}=C_{\tau}(\mathbb{N})$, $\mathcal{T}$ being the weak* topology on the spectrum of $\mathfrak{A}$.\\
Thanks to proposition \ref{prop1}, we have $C_{\tau}(\mathbb{N})\cong l_1$;
since no proper inclusion relationships are allowed between preduals, we finally get $M=\mathfrak{A}$. This concludes the proof.
\end{proof}

$$$$
\textsl{DIP. MAT. CASTELNUOVO, UNIV. DI ROMA LA SAPIENZA, ROME, ITALY}\\
\emph{E-mail address:} \verb"s-rossi@mat.uniroma1.it"
\end{document}